\documentclass[10pt,english]{amsart}
\usepackage[T1]{fontenc}
\usepackage[latin1]{inputenc}
\usepackage{color}
\usepackage{amstext}
\usepackage{amsthm}
\usepackage{amssymb}

\makeatletter

\providecommand{\tabularnewline}{\\}

\numberwithin{equation}{section}
\numberwithin{figure}{section}
\theoremstyle{plain}
\newtheorem{thm}{\protect\theoremname}
  \theoremstyle{plain}
  \newtheorem{prop}[thm]{\protect\propositionname}
  \theoremstyle{plain}
  \newtheorem{lem}[thm]{\protect\lemmaname}
  \theoremstyle{remark}
  \newtheorem{rem}[thm]{\protect\remarkname}
  \theoremstyle{definition}
  \newtheorem{example}[thm]{\protect\examplename}
 \theoremstyle{definition}
 \newtheorem*{defn*}{\protect\definitionname}

\makeatother

\makeatother

\usepackage{babel}
  \providecommand{\definitionname}{Definition}
  \providecommand{\examplename}{Example}
  \providecommand{\lemmaname}{Lemma}
  \providecommand{\propositionname}{Proposition}
  \providecommand{\remarkname}{Remark}
\providecommand{\theoremname}{Theorem}

\begin{document}

\title[On generalized Kummer surfaces]{On  generalized Kummer surfaces and the orbifold Bogomolov-Miyaoka-Yau inequality}

\addtolength{\textwidth}{0mm}
\addtolength{\hoffset}{-0mm} 
\addtolength{\textheight}{0mm}
\addtolength{\voffset}{-0mm} 



\global\long\global\long\def\Alb{{\rm Alb}}
 \global\long\global\long\def\Jac{{\rm Jac}}
\global\long\global\long\def\Disc{{\rm Disc}}
\global\long\global\long\def\Tr{{\rm Tr}}
 \global\long\global\long\def\NS{{\rm NS}}
\global\long\global\long\def\PicVar{{\rm PicVar}}
\global\long\global\long\def\Pic{{\rm Pic}}
\global\long\global\long\def\Br{{\rm Br}}
 \global\long\global\long\def\Pr{{\rm Pr}}
\global\long\global\long\def\Km{{\rm Km}}
\global\long\global\long\def\rk{{\rm rk}}
\global\long\global\long\def\Hom{{\rm Hom}}
\global\long\global\long\def\Div{{\rm div}}
 \global\long\global\long\def\End{{\rm End}}
 \global\long\global\long\def\aut{{\rm Aut}}
 \global\long\global\long\def\SSm{{\rm S}}
 \global\long\global\long\def\psl{{\rm PSL}}
 \global\long\global\long\def\CC{\mathbb{C}}
 \global\long\global\long\def\BB{\mathbb{B}}
 \global\long\global\long\def\PP{\mathbb{P}}
 \global\long\global\long\def\QQ{\mathbb{Q}}
 \global\long\global\long\def\RR{\mathbb{R}}
 \global\long\global\long\def\FF{\mathbb{F}}
 \global\long\global\long\def\DD{\mathbb{D}}
 \global\long\global\long\def\NN{\mathbb{N}}
 \global\long\global\long\def\ZZ{\mathbb{Z}}
 \global\long\global\long\def\HH{\mathbb{H}}
 \global\long\global\long\def\Gal{{\rm Gal}}
 \global\long\global\long\def\OO{\mathcal{O}}
 \global\long\global\long\def\pP{\mathfrak{p}}
 \global\long\global\long\def\pPP{\mathfrak{P}}
 \global\long\global\long\def\qQ{\mathfrak{q}}
 \global\long\global\long\def\mm{\mathcal{M}}
 \global\long\global\long\def\aaa{\mathfrak{a}}
\global\long\def\a{\alpha}
\global\long\def\b{\beta}
 \global\long\def\d{\delta}
 \global\long\def\D{\Delta}
\global\long\def\L{\Lambda}
 \global\long\def\g{\gamma}
 \global\long\def\G{\Gamma}
 \global\long\def\d{\delta}
 \global\long\def\D{\Delta}
 \global\long\def\e{\varepsilon}
 \global\long\def\k{\kappa}
 \global\long\def\l{\lambda}
 \global\long\def\m{\mu}
 \global\long\def\o{\omega}
 \global\long\def\p{\pi}
 \global\long\def\P{\Pi}
 \global\long\def\s{\sigma}
 \global\long\def\S{\Sigma}
 \global\long\def\t{\theta}
\global\long\def\qa{\frak{a}}
 \global\long\def\T{\Theta}
 \global\long\def\f{\varphi}
 \global\long\def\deg{{\rm deg}}
 \global\long\def\det{{\rm det}}
 \global\long\def\dps{{\displaystyle }}
 \global\long\def\Dem{D\acute{e}monstration: }
 \global\long\def\ker{{\rm Ker\,}}
 \global\long\def\im{{\rm Im\,}}
 \global\long\def\rg{{\rm rg\,}}
 \global\long\def\car{{\rm car}}
\global\long\def\fix{{\rm Fix( }}
 \global\long\def\card{{\rm Card\  }}
 \global\long\def\codim{{\rm codim\,}}
 \global\long\def\coker{{\rm Coker\,}}
 \global\long\def\mod{{\rm mod }}
 \global\long\def\pgcd{{\rm pgcd}}
 \global\long\def\ppcm{{\rm ppcm}}
 \global\long\def\la{\langle}
 \global\long\def\ra{\rangle}

\subjclass[2000]{Primary 14J28, 14L30, 32J25, Secondary 14J50}

\keywords{Generalized K3 surfaces, Kummer Lattices, Quotient surfaces, orbifold
Bogomolov-Miyaoka-Yau inequality }

\author{Xavier Roulleau}
\begin{abstract}
A generalized Kummer surface $X=\Km(T,G)$ is the resolution of a
quotient of a torus $T$ by a finite group of symplectic automorphisms
$G$. We complete the classification of generalized Kummer surfaces
by studying the two last groups which have not been yet studied. For
these surfaces we compute the associated Kummer lattice $K_{G}$,
which is the minimal primitive sub-lattice containing the exceptional
curves of the resolution $X\to T/G$. We then prove that a K3 surface
is a generalized Kummer surface of type $\Km(T,G)$ if and only if
its Néron-Severi group contains $K_{G}$. \\
For smooth-orbifold surfaces $\mathcal{X}$ of Kodaira dimension $\geq0$,
Kobayashi proved the orbifold Bogomolov Miyaoka Yau inequality $c_{1}^{2}(\mathcal{X})\leq3c_{2}(\mathcal{X}).$
For Kodaira dimension $2$, the case of equality is characterized
as $\mathcal{X}$ being uniformized by the complex $2$-ball $\BB_{2}$.
For smooth-orbifold K3 and Enriques surfaces we characterize the case
of equality as being uniformized by $\CC^{2}$. 
\end{abstract}

\maketitle

\section{Introduction}

A K3 surface $X$ is called a \textit{generalized Kummer surface}
and we write $X=\Km(T,G)$ if it is the resolution of a quotient $T/G$
where $T$ is a torus and $G$ is a finite group of automorphisms
of $T$. Let $X=\Km(T,G)$ be a generalized Kummer surface and let
$F_{G}$ be the sub-lattice of the Néron Severi group $\NS(X)$ generated
by the exceptional divisors $\mathcal{C}_{G}$ of the resolution $X\to T/G$.
The minimal primitive sub-lattice $K_{G}$ of $\NS(X)$ containing
the lattice $F_{G}$ is called the \textit{Kummer lattice} of $G$.
\\
In \cite{Nikulin} Nikulin computes the Kummer lattice $K_{\ZZ/2\ZZ}$
and obtains the famous result that for a K3 surface $X$, it is equivalent
to be a Kummer or to contain $16$ disjoint $(-2)$-curves or that
there exists a primitive embedding of the lattice $K_{\ZZ/2\ZZ}$
in the Néron-Severi group of $X$.\\
That result linking the primitive embedding of a lattice $K_{G}$
contained in $\NS(X)$ to a geometric description of $X$ has then
been extended by Bertin \cite{Bertin} and Garbagnati \cite{Garbagnati}
to the $7$ other symplectic automorphism groups $G$ acting on some
torus $T$ such that the action of $G$ preserves the origin of $T$. 

It turns out that there are symplectic groups $G$ which has not been
yet studied: when $G$ has no global fix points on the torus $T$.
Up to taking quotient of $G$ by its translation subgroup, one can
suppose that $G$ contains no translations. Fujiki has described and
classified such pairs $(T,G)$: then $G$ is isomorphic to $\hat{Q}_{8}$
(the quaternion group) or $\hat{T}_{24}$ (the binary tetrahedral
group of order $24$). We compute that the singularities of the quotient
$T/\hat{Q}_{8}$ are $\mathcal{C}_{\hat{Q}_{8}}=6A_{3}+A_{1}$ and
the ones of $T/\hat{T}_{24}$ are $\mathcal{C}_{\hat{T}_{24}}=4A_{2}+2A_{3}+A_{5}$.
Let be $G=\hat{Q}_{8}$ or $\hat{T}_{24}$. In sub-sections \ref{subsec:Q8}
and \ref{subsec:The-configuration Q24}, we describe the minimal primitive
sub-lattice $K_{G}$ containing the lattice $F_{G}$ generated by
the exceptional curves $\mathcal{C}_{G}$ of the minimal resolution
of $T/G$, and we obtain the following result:
\begin{thm}
Let $X$ be a K3 surface and $G$ be the group $\hat{Q}_{8}$ or $\hat{T}_{24}$.
The following conditions are equivalent:\\
i) $X$ is a generalized Kummer surface $X=\Km(T,G)$ \\
ii) the Kummer lattice $K_{G}$ is primitively embedded in $\NS(X)$.\\
iii) $X$ contains a configuration of $ADE$ curves $\mathcal{C}_{G}$. 
\end{thm}
Then we turn our attention to a related question, which was our initial
motivation. Let $\mathcal{C}$ be a configuration of disjoint $ADE$
curves on a smooth surface $X$ and let $X\to\mathcal{X}$ be the
contraction of the connected components of $\mathcal{C}$. To the
singular surface $\mathcal{X}$ one can associate its orbifold Chern
numbers, denoted by $c_{1}^{2}(\mathcal{X}),\,c_{2}(\mathcal{X})\in\QQ,$
which depends on the Chern numbers $c_{1}^{2}(X),\,c_{2}(X)$ of $X$
and on the number and type of the $ADE$ singularities of $\mathcal{X}$.
These orbifold Chern numbers have the following property:
\begin{thm}
\label{thm:(Kobayashi)-1} (Orbifold Bogomolov-Miyaoka-Yau inequality,
\cite{Kobayashi, Kobayashi2, Hirzebruch, Megyesi}).\footnote{Note that there exist stronger versions of Theorem \ref{thm:(Kobayashi)-1},
in particular with other quotient singularities, but for surfaces
of Kodaira dimension $0$ which is the case of interest for us, the
only quotient singularities one can obtain are $ADE$.} Suppose that $X$ is a minimal algebraic surface of Kodaira dimension
$\geq0$. Then:\\
A) One has 
\begin{equation}
c_{1}^{2}(\mathcal{X})\leq3c_{2}(\mathcal{X}).\label{eq:BMY-1}
\end{equation}
B) Suppose $X$ has general type. Equality holds in \ref{eq:BMY-1}
if and only if there exists a discrete cocompact lattice $\Gamma$
in $PU(2,1)$ such that $\mathcal{X}=\BB_{2}/\Gamma$. In other words,
one has equality if and only if $\mathcal{X}$ is uniformisable by
the unit ball $\BB_{2}$.
\end{thm}
Here a discrete cocompact lattice means a subgroup which is discrete
in $PU(2,1)$, such that the points with non-trivial isotropy are
isolated, these isotropy groups are finite, and the quotient $\BB_{2}/\G$
is compact. A consequence of Theorem \ref{thm:(Kobayashi)-1} is that
in case of equality in \eqref{eq:BMY-1}, there always exists a finite
uniformisation of $\mathcal{X}$ i.e. a smooth ball quotient surface
$Z$ having a finite group of automorphisms $G$ such that $\mathcal{X}=Z/G$.

It is now natural to ask if there is an analog of part B) of Theorem
\ref{thm:(Kobayashi)-1} for surfaces of Kodaira dimension $0$ and
$1$. In this paper we study that problem for surfaces having Kodaira
dimension $\kappa=0$, for which equality $c_{1}^{2}(\mathcal{X})=3c_{2}(\mathcal{X})$
is in fact equivalent to $c_{2}(\mathcal{X})=0$. Let $X$ be a K3
surface, let $\mathcal{C}$ be a configuration of $ADE$ curves on
$X$ and let $X\to\mathcal{X}$ be the contraction of the curves in
$\mathcal{C}$. We obtain the following result:
\begin{thm}
\label{thm:MAIN-1} The equality $c_{1}^{2}(\mathcal{X})=3c_{2}(\mathcal{X})$
holds if and only if there exists a discrete cocompact lattice $\Gamma$
in the affine linear group $\CC^{2}\rtimes GL_{2}(\CC)$ such that
$\mathcal{X}=\CC^{2}/\Gamma$. 
\end{thm}
Using the now complete classification of generalized Kummer surfaces,
we will in fact see that in case of equality $c_{1}^{2}(\mathcal{X})=3c_{2}(\mathcal{X})$,
the $K3$ $X$ is a generalized Kummer surface, which result implies
Theorem \ref{thm:MAIN-1}. 

We then obtain the same result as Theorem \ref{thm:MAIN-1} for Enriques
surfaces: for an Enriques surface $X$ and $X\to\mathcal{X}$ the
contraction of a configuration $\mathcal{C}$ of $ADE$ curves, one
has $c_{1}^{2}(\mathcal{X})=3c_{2}(\mathcal{X})$ if and only if $\mathcal{C}$
is the union of $8$ disjoint $(-2)$-curves. We moreover construct
the Enriques surfaces containing such a configuration. \\
Among algebraic surfaces with Kodaira dimension $0$, it remains the
Abelian and bi-elliptic surfaces, which satisfy $c_{1}^{2}=3c_{2}=0$.
The universal cover of these surfaces is $\CC^{2}$, and they do not
contain rational curves; therefore the question is closed for surfaces
with $\kappa=0$.

The paper is organized as follows: in Section \ref{sec:Preliminaries}
we recall the notations and the main results we will need, which are
mainly the results of Garbagnati \cite{Garbagnati}. In section \ref{sec:Classification-of-symplectic},
we recall Fujiki's beautiful classification of automorphism groups
of $2$-dimensional tori and we give a more detailed account of the
previous work on generalized Kummer surfaces. Then we describe the
Kummer lattices $K_{G}$ for $G=\hat{Q}_{8}$, $\hat{T}_{24}$, and
prove that if a Kummer surface contains a configuration $\mathcal{C}_{G}$
then it is a generalized Kummer surface $\Km(T,G)$. In section \ref{sec:Kobayashi's-problem-for},
we prove Theorem \ref{thm:MAIN-1} on K3 and Enriques surfaces.

\textbf{Acknowledgements.} The author wishes to thank the referee
for careful reading of the manuscript and useful remarks.

\section{Preliminaries\label{sec:Preliminaries}}

\subsection{Notations}

$\ZZ_{n}=\ZZ/n\ZZ$ is the cyclic group of order $n$. \\
$Q_{8}$ is the quaternion group (order $8,$ has a unique involution
$\iota$, $Q_{8}/\iota\simeq(\ZZ_{2})^{2}$). \\
$D_{12}$ is the binary dihedral group (order $12$).\\
$T_{24}$ is the binary tetrahedral group (order $24$, isomorphic
to $SL_{2}(\FF_{3})$; $Q_{8}$ is a normal subgroup of it).

For $n\in\ZZ$, $[n]:T\to T$ is the multiplication by $n$ map on
a torus $T$. 

For more on the problematic of generalized Kummer surfaces, we recommend
the paper of Garbagnati \cite{Garbagnati}, from which we tried to
follow the notations.

\subsection{Lattices, divisible sets}

For a lattice $L$, we denote by $L^{\vee}$ its dual. The\textit{
length} of $L$ is the minimal number of generators of its \textit{discriminant
group} $L^{\vee}/L$. A sublattice $M$ of $L$ is said \textit{primitive}
if $L/M$ is torsion free. 
\begin{prop}
\label{prop:()-NikuLattice}(\cite[Proposition 1.6.1]{Nikulin2})
Let $L$ be a unimodular lattice, $M$ be a primitive sublattice of
$L$ and $M^{\perp}$ the orthogonal to $M$ in $L$. The discriminant
group of $M$ is isomorphic to the discriminant group of $M^{\perp}$.
In particular, since the length of a lattice is at most the rank of
the lattice, $l(M)=l(M^{\perp})\leq\min(\rk(M),\rk(M^{\perp}))$. 
\end{prop}
A set of disjoint smooth rational curves $(C_{i})_{i\in I}$ on a
surface $X$ is called\textit{ even }if there exist an invertible
sheaf $\mathcal{L}$ such that $\OO(\sum_{i}C_{i})=\mathcal{L}^{\otimes2}$
. On a K3 surface, an even set contains $8$ or $16$ disjoint curves.
\\
Let $C_{i}^{j},\,1\leq i\leq n,\,j\in\{1,2\}$ be a set of $n$ disjoint
$A_{2}$ configurations (so that $C_{i}^{1}C_{i}^{2}=1$) on a surface.
The divisor $D=\sum_{i=1}^{n}C_{i}^{1}+2C_{i}^{2}$ is called \textit{$3$-divisible}
if $\OO(D)=\mathcal{L}^{\otimes3}$ where $\mathcal{L}$ is an invertible
sheaf. On a K3 surface, the support of a $3$-divisible divisor contains
$6$ or $9$ disjoint $A_{2}$ configurations.\\
 We will use repeatedly the following consequence of Proposition \ref{prop:()-NikuLattice}:
\begin{lem}
\label{lem:8 nodes}Let $X$ be a K3 surface containing $12$ disjoint
$(-2)$-curves. Then there exists an even set supported on $8$ of
these curves. \\
Let $X$ be a K3 surface containing $13$ disjoint $(-2)$-curves.
Then there exist two linearly independent even sets of curves, supported
on $12$ of these curves.
\end{lem}
\begin{proof}
This is well known, see eg \cite[Remark 8.10]{GS}. The discriminant
group of the lattice generated by the $12$ curves is $(\ZZ_{2})^{12}$,
it has length $12>\min(12,22-12)$, thus there exist at least one
non-trivial divisible class.\\
The second part follows e.g. from the first: there is an even set,
then remove one curve from that even set, there exists still a set
of $12$ disjoint curves, thus another even set, with a different
support.
\end{proof}
In the present paper an \textit{$ADE$ configuration} $\mathcal{C}$
on a surface will have a polysemic meaning. It could mean a set of
$ADE$ singularities on a surface $\mathcal{X}$, or the set of the
exceptional curves of its minimal resolution $X\to\mathcal{X}$. \\
For numbers $\a_{n},\d_{n},\e_{n}\in\NN^{*}$ with $\d_{i}=0$ for
$i\leq3$ and $\e_{i}=0$ for $i\notin\{6,7,8\}$, we write symbolically
$\mathcal{C}=\sum_{n\geq1}\a_{n}A_{n}+\d_{n}D_{n}+\e_{n}E_{n}$ if
for any $n\geq1$, $\mathcal{C}$ contains $\a_{n}$ (resp. $\d_{n}$,
$\e_{n}$) configurations of type $A_{n}$ (resp. $D_{n},E_{n}$).

Let $\mathcal{C}$ be a $ADE$ configuration. We are looking for obstructions
or criteria for some sub-configurations of $\mathcal{C}$ to be part
of an even set or a $3$-divisible set. In Remark \ref{rem:necessary cond even}
below, when we speak of a configuration $A_{n}$ or $D_{n}$, we implicitly
assume it is maximal in $\mathcal{C}$, i.e. it is not contained in
a $A_{m}$ or $D_{m}$ contained in $\mathcal{C}$ for some $m>n$.
\begin{rem}
\label{rem:necessary cond even}A) An irreducible component $C$ of
a configuration $A_{n}$ in $\mathcal{C}$ can be part of an even
set $E$ if and only if $n$ is odd, the $\frac{n+1}{2}$ disjoint
curves in that $A_{n}$ configuration are in $E$ and $C$ is among
these curves, since otherwise there always exists a curve $C'$ supported
on $A_{n}$ such that $C'E=1$, and therefore $E$ cannot be even.\\
B) The discriminant group of $D_{n}$ is $(\ZZ_{2})^{2}$ if $n\geq4$
is even and is $\ZZ_{4}$ if $n$ is odd (see \cite[Theorem 2.3.5]{Griess}).
Accordingly, $k$ disjoint curves and the $2$ extremal disjoint closest
curves on a $D_{2k}$ can be possibly part of an even set; the two
closest extremal disjoint curves on a $D_{2k+1}$ can be possibly
part of an even set. \\
C) The discriminant group of $A_{n}$ is $\ZZ_{n+1}$.  Since $\ZZ_{4}$
does not contain the group $\ZZ_{3}$, a sub-configuration $A_{2}$
of a configuration $A_{3}$ in $\mathcal{C}$ cannot be part of a
$3$-divisible set. There is no such an obstruction for the two disjoint
$A_{2}$ in a configuration $A_{5}$.
\end{rem}

\subsection{Double, bi-double and triple covers, lifts of automorphisms}

To an even set $E$ (resp. a $3$-divisible divisor $E=\sum_{i=1}^{n}C_{i}^{1}+2C_{i}^{2}$)
on a K3 surface $X$, one can associate a double (resp. triple) cyclic
cover of $X$ branched on the support of $E$. The minimal desingularisation
$Y$ of that cyclic cover has an involution (resp. an automorphism
of order $3$) $\tau$ such that $Y/\tau$ is (isomorphic to) $\mathcal{X}$,
the surface obtained by contracting the curves on the support of $E$.
We call $Y$ the \textit{surface associated} to $E$. 
\begin{lem}
\label{lem:An-automorphism-lift}Let $E$ be an even set on a K3 surface
$X$ and let $Y$ be the surface associated to $E$.\\
A) An automorphism $\s$ of order $n$ of the K3 surface $X$ lifts
to $Y$ if and only if $E=\s^{*}E$. \\
B) Suppose that $\s$ lifts to an automorphism $\s'\in\aut(Y)$. Let
$\tau$ be an element of the transformation group of the cover $Y$
(thus $Y/\tau$ is birational to $X$). There is an exact sequence
\[
0\to\la\tau\ra\to\la\tau,\s'\ra\to\la\s\ra\to1.
\]
\end{lem}
\begin{proof}
A K3 surface satisfies $\NS(X)=\Pic(X)$, then part A) is \cite[Proposition 4.2]{PRR}.
Part B) follows from the fact $\s'\tau\s'^{-1}$ is a lift of the
identity, thus a power of $\tau$ and $\la\tau\ra$ is normal in $\la\tau,\s'\ra$;
the map $\la\tau,\s'\ra\to\la\s\ra$ maps a lift $\mu'$ of $\mu\in\aut(X)$
to $\mu$.
\end{proof}
A bi-double cover $Y\to X$ of a surface $X$ is a Galois cover of
group $(\ZZ_{2})^{2}$. It is determined by divisors $D_{1},D_{2},D_{3}$
and invertible sheaves $L_{1},L_{2},L_{3}$ such that for $\{i,j,k\}=\{1,2,3\}$,
one has 
\begin{equation}
2L_{i}\equiv D_{j}+D_{k},\label{eq:relations even}
\end{equation}
(see \cite{Catanese}). The surface $Y$ is embedded in the total
space of the vector bundle $\mathbb{L}=L_{1}\oplus L_{2}\oplus L_{3}$
as the variety with equation 
\[
\rk\left(\begin{array}{ccc}
x_{1} & w_{3} & w_{2}\\
w_{3} & x_{2} & w_{1}\\
w_{2} & w_{1} & x_{3}
\end{array}\right)=1,
\]
where $D_{i}=\Div(x_{i})$ and $w_{1},w_{2},w_{3}$ are coordinates
of the $L_{i}$. 
\begin{example}
Let $E=\sum_{i=1}^{12}C_{i}$ be a $12A_{1}$ configuration on a K3
surface, such that $E$ has $2$ linearly independent even sets $\ell_{1},\,\ell_{2}$.
Up to reordering, one can suppose 
\[
\ell_{1}=\sum_{i=5}^{12}C_{i},\,\,\ell_{2}=\sum_{i=1}^{4}C_{i}+\sum_{i=9}^{12}C_{i}.
\]
Then $\ell_{3}=\sum_{i=1}^{8}C_{i}$ is also even. Let be $L_{i}:=\frac{1}{2}\ell_{i}$
and $D_{j}=E-\ell_{j}$. The datas $D_{i},\,L_{j}$ satisfy the relations
\ref{eq:relations even} and determine a bi-double cover $Y\to X$;
$Y$ is a smooth K3. 
\end{example}
Let $\s$ be an automorphism of a smooth surface $X$ admitting a
bi-double cover determined by divisors $D_{i},\,i\in\{1,2,3\}$ and
invertible sheaves $L_{i},\,i\in\{1,2,3\}$ as above. Suppose that
there is an action of $\sigma$ on $\{1,2,3\}$ such that $\s^{*}L_{i}=L_{\s i}$
and $\s^{*}D_{i}=D_{\s i}$. Then
\begin{lem}
\label{lem:lift order 3 to bidoubles}The automorphism $\s$ lifts
to an automorphism of $Y$. 
\end{lem}
\begin{proof}
One can choose coordinates $w_{i}$ so that $\s^{*}w_{i}=w_{\s i}$
and equations $x_{i}$ of $D_{i}$ such that $\s^{*}x_{i}=x_{\s i}$.
Then the automorphism $\s$ lifts to an automorphism of $\mathbb{L}$,
and the equations of $Y$ are preserved, thus it restricts to an automorphism
of $Y$. 
\end{proof}

\subsection{Roots of a lattice and $(-2)$-curves}

Let $X$ be an algebraic K3 surface. Let $h$ be a pseudoample divisor
on $X$ (i.e. $h^{2}>0$ and $hD\geq0$ for all effective divisor
$D$) and let 
\[
L=h^{\perp}:=\{l\in\NS(X)\text{ such that }lh=0\}
\]
be the orthogonal of $h$ in $\NS(X)$. We will use the following
result proved by Garbagnati \cite[Proposition 3.2]{Garbagnati} (see
also \cite[Lemma 3.1]{Bertin} of Bertin):
\begin{prop}
\label{prop: Bertin-Garbagati} Let us assume that there exists a
root lattice $R$ such that: \\
(1) $L$ is an overlattice of finite index of $R$ \\
(2) the roots of $R$ and of $L$ coincide. \\
Then there exists a basis of $R$ which is supported on smooth irreducible
rational curves. 
\end{prop}
\begin{rem}
According to a recent preprint of Schütt \cite{Schutt}, hypothesis
(2) is always satisfied.
\end{rem}
Let $X$ be a K3 surface and $F\subset\NS(X)$ be a sub-lattice. A
minimal primitive sub-lattice of $H^{2}(X,\ZZ)$ containing $F$ is
a lattice $K_{F}$ containing $F$ such that $K_{F}/F$ is finite
and $H^{2}(X,\ZZ)/K_{F}$ is free. That lattice $K_{F}$ is unique
and is equal to the lattice $\NS(X)\cap F\otimes\QQ$.

Let $X$ be a non-algebraic K3 surface. By Grauert's ampleness criterion
for complex surfaces, since $X$ is not algebraic any divisor on $X$
has self intersection $\leq0$. Therefore the irreducible curves are
$-2$-curves or of arithmetic genus $1$. By Riemann-Roch, if there
is a curve of arithmetic genus $1$, then it is a fiber of a fibration
$X\to\PP^{1}$. That fibration is then unique and contracts every
$(-2)$-curves on $X$ (this is still because the divisors have self-intersection
$\leq0$). The class of a fiber generates the kernel of the natural
map $\NS(X)\to\text{Num}(X)$, in particular $\NS(X)$ is degenerate.
Therefore since the signature of the intersection form on $H^{1,1}$
is $(1,19)$, if $\NS(X)$ contains a negative definite sub-group
of rank $19$ then there is no curves of arithmetic genus $1$ on
the non algebraic K3 surface $X$.
\begin{rem}
\label{rem:SimpleBase}Let $X$ be a non algebraic K3 surface containing
no curves of arithmetic genus $1$. The negative definite lattice
$\NS(X)$ has rank $\rho\leq19$ and $\rho$ is equal to the number
of $(-2)$-curves on $X$. In particular the minimal primitive sub-lattice
containing the $(-2)$-curves on $X$ is $\NS(X)$ itself.\\
Let $\d$ be a $(-2)$-class on $X$ i.e. an element of $\NS(X)$
such that $\d\in NS(X)$ satisfies $\d^{2}=-2$. By Riemann-Roch $-\d$
or $\d$ is effective; say $\d\geq0$ : there exists $(-2)$-curves
$C_{i}$ such that $\d=\sum m_{i}C_{i}$, with $m_{i}\geq1$. Therefore
the $(-2)$-classes on $X$ form a root system of the lattice $F$
generated by the $(-2)$-curves on $X$ and the $(-2)$-curves form
a simple base $B$ of $F$. One can then apply \cite[Lemma 3.2]{Bertin}
(see also \cite[Remark 4.5]{Garbagnati}) and conclude that if one
has a direct sum decomposition as root lattice $F=\sum_{n\geq1}A_{n}^{\oplus\a_{n}}$,
then for each of the factors $A_{n}$ there is a simple base constituted
of $(-2)$-curves.
\end{rem}

\subsection{Orbifold settings}

Let $\mathcal{C}=\sum_{n\geq1}\a_{n}A_{n}+\d_{n}D_{n}+\e_{n}E_{n}$
be a $ADE$ configuration of curves on a smooth surface $X$. Let
us define the quantity
\[
m(\mathcal{C}):=\sum_{n\geq1}(\a_{n}+\d_{n}+\e_{n})(n+1)-\sum_{n\geq1}\frac{\a_{n}}{n+1}-\sum_{n\geq4}\frac{\d_{n}}{4(n-2)}-\frac{\e_{6}}{24}-\frac{\e_{7}}{48}-\frac{\e_{8}}{120}.
\]
Let $X\to\mathcal{X}$ be the contraction map of the curves contained
in $\mathcal{C}$. Since $\mathcal{X}$ contains only $ADE$ singularities,
the orbifold Chern numbers of $\mathcal{X}$ are 
\[
c_{1}^{2}(\mathcal{X})=K_{X}^{2},\,\text{ and }c_{2}(\mathcal{X})=c_{2}(X)-m(\mathcal{C}),
\]
(see e.g. \cite{RR}). Let $X$ be a K3 surface. The orbifold Miyaoka-Yau
inequality \ref{eq:BMY-1} tells us that 
\[
m(\mathcal{C})\leq24.
\]
Moreover, since each configuration $A_{n},D_{n}$ or $E_{n}$ contributes
for a $n$-dimensional subspace in the negative definite part (of
rank at most $19$) of the Néron-Severi group, one has the restriction
\[
\sum_{n\geq1}n(\a_{n}+\d_{n}+\e_{n})\leq19.
\]
Suppose that the $K3$ orbifold $\mathcal{X}$ has a finite uniformization
$Y\to\mathcal{X}$, i.e. $Y$ is a smooth surface with an action by
a finite group $G$ of order $n$ such that $\mathcal{X}=Y/G$ and
$Y\to Y/G$ is ramified in codimension $2$. Then 
\begin{lem}
The Chern numbers of $Y$ are $K_{Y}^{2}=nc_{1}^{2}(\mathcal{X})$
and $c_{2}(Y)=nc_{2}(\mathcal{X})$. The surface $Y$ is Abelian or
a K3.
\end{lem}
\begin{proof}
The first part follows by the definition of the orbifold Chern numbers,
see e.g. \cite{Kobayashi, RR}. Since $Y\to\mathcal{X}$ is ramified
in codimension 2, the canonical divisor $K_{Y}$ is the pull-back
of $K_{\mathcal{X}}$, which is trivial, thus $K_{Y}$ is trivial
and $Y$ is a K3 or is an Abelian surface.
\end{proof}
The double cover of an even set of $8$ (resp. $16$) $A_{1}$ is
a $K3$ (resp. a torus). The triple cover of a $3$-divisible set
of $6$ (resp. $9$) $A_{2}$ is a $K3$ (resp. a torus).

\section{Classification of symplectic groups and generalized Kummer surfaces\label{sec:Classification-of-symplectic}}

\subsection{Fujiki's constructions of Abelian tori with symplectic action of
a group\label{subsec:Fujiki's-constructions}}

In \cite{Fujiki}, Fujiki constructs and classifies pairs ($T,G)$
of complex tori $T$ with a faithful action by a group $G$ containing
no translations. Let us describe his results when $G$ acts symplectically
and is not cyclic. 

Let $\mathbb{H}=\RR[1,i,j,k]$ be the quaternion field, so that 
\[
i^{2}=j^{2}=k^{2}=-1,\,ij=-ji=k.
\]
Let 
\[
\qa=\ZZ[1,i,j,t]
\]
be the ring of Hurwitz quaternions, where $t=\frac{1}{2}(1+i+j+k)$.
This is a maximal order of $F=\QQ[1,i,j,k]$ and its group of invertible
elements is
\[
\qa^{\times}=\{1,\pm i,\pm j,\pm k,\frac{1}{2}(\pm1\pm i\pm j\pm k)\},
\]
which is the binary tetrahedral group $T_{24}$. Let be 
\[
\qa_{0}=\ZZ[1,i,j,k];
\]
the subgroup $\qa_{0}^{\times}=\{1,\pm i,\pm j,\pm k\}$ is the quaternion
group $Q_{8}$. Let be 
\[
F'=\QQ[1,i,\sqrt{3}j,\sqrt{3}k],
\]
and $\frak{b}=\ZZ[1,i,h,l]$, where 
\[
h=\frac{1}{2}(i+\sqrt{3}j),\,\,l=\frac{1}{2}(1+\sqrt{3}k).
\]
The subgroup 
\[
\frak{b}^{\times}=\{\pm1,\pm i,\pm h,\pm l,\pm ih,\pm il\}
\]
is the binary dihedral group $D_{12}$ of order $12$.

Let us define the following lattices in $\HH$:
\[
\L_{Q_{8}}=\qa_{0},\,\,\L_{D_{12}}=\frak{b},\,\,\L_{T_{24}}=\qa.
\]
 The set $\mathcal{X}$ of pure quaternions: 
\[
\mathcal{X}=\{q\in\mathbb{H}\,|\,q^{2}=-1\}=\{ai+bj+ck\,|\,a^{2}+b^{2}+c^{2}=1\},
\]
is isomorphic to $\PP_{\CC}^{1}$. For $q\in\mathcal{X}$, one can
identify $\RR+q\RR$ with $\CC$ by sending $q$ to $\sqrt{-1}$.
By multiplication on the right, $\mathcal{X}$ parametrizes complex
structures of $\mathbb{H}=F\otimes\RR=\RR^{4}$. \\
For $G=Q_{8},\,D_{12}$ or $T_{24}$, such a complex structure induces
a complex structure on the real torus $T_{q}:=\HH/\L_{G}$. The left
multiplication on $\HH$ induces a left action of $G=\L_{G}^{\times}$
on $T_{q}=\HH/\L_{G}$, which is compatible with the complex structure
induced by $q$, in other words that action is holomorphic. In that
way we get a holomorphic family of pairs 
\[
(T_{q},G)_{q\in\mathcal{X}}
\]
of a complex tori $T_{q}$ with an action of the group automorphism
$G$ (preserving $0\in T_{q}$), parametrized by $q\in\mathcal{X}_{G}=\mathcal{X}\simeq\PP^{1}$. 

We say that a group acts \textit{symplectically} on a torus (or is
\textit{symplectic}) if its analytic representation is in $SL_{2}(\CC)\subset GL_{2}(\CC)$.
According to \cite{Fujiki}, the three groups $G=$$Q_{8},D_{12}$
and $T_{24}$ acts symplectically on the torus $T_{q}=\HH/\L_{G}$. 
\begin{defn*}
We say that two pairs $(T_{1},G_{1})$ and $(T_{2},G_{2})$ of tori
$T_{1},T_{2}$ with action by groups $G_{1},G_{2}$ are \textit{isomorphic}
if there is an isomorphism of $T_{1}$ with $T_{2}$ such that the
action of $G_{1}$ on $T_{2}$ (induced by transport of structure)
is $G_{2}$ (in particular $G_{1}\simeq G_{2}$). We say that a symplectic
group $G$ acting on a torus $T$ is \textit{reduced} if it contains
no translations.
\end{defn*}
Let $G$ be a symplectic group of automorphism of a torus $T$ and
let $G_{0}$ be its subgroup of translations.
\begin{lem}
The group $G_{0}$ is normal in $G$ and $G/G_{0}$ is a reduced symplectic
group of automorphisms of the torus $T/G_{0}$.
\end{lem}
\begin{proof}
It is easy to check that $G_{0}$ is normal (the translation subgroup
of a torus is normal). The quotient $T/G_{0}$ is of course a torus
; the group $G/G_{0}$ acts on $T/G_{0}$ symplectically since the
analytic representations of an element in $G$ or its image in $G/G_{0}$
are the same.
\end{proof}
We say that a finite reduced group $G$ is \textit{maximal} if $G$
is not a strict sub-group of another reduced finite symplectic group.
Let $G$ be a non-cyclic group of symplectic automorphisms of a torus
$T$, fixing one point globally (which we can suppose to be the origin;
that hypothesis implies that $G$ is reduced). We have
\begin{thm}
(Fujiki \cite[Proposition 3.5 and Theorem 3.11]{Fujiki}) The group
$G$ is isomorphic to one of the groups $Q_{8},\,D_{12}$ or $T_{24}$.\\
 If $G$ is maximal, then there exists $q\in\mathcal{X}$ such that
$(T,G)$ is isomorphic to $(T_{q},G)$, where $T_{q}=\HH/\L_{G}$
with complex structure given by $q$. \\
If $G$ is not maximal, then $G=Q_{8}$ and there exists $q\in\mathcal{X}$
such that $(T,G)$ is isomorphic to $(T_{q},Q_{8})$, where $T_{q}=\HH/\L_{T_{24}}$
and $Q_{8}\subset T_{24}$ is the unique quaternion group of order
$8$ contained in $T_{24}$. 
\end{thm}
For $q\in\mathcal{X}$ and $T_{q}=\HH/\L_{T_{24}}$, let us now denote
by $\mathbb{A}(T_{q})$ and $\mathbb{A}_{0}(T_{q})$ respectively
the group of real affine automorphisms and the group of translations
of $T_{q}$. Then $\mathbb{A}(T_{q})$ is naturally a semi-direct
product $\mathbb{A}(T_{q})=\aut_{\ZZ}\L_{T_{24}}\ltimes\mathbb{A}_{0}(T_{q})$.
Let $\l\in\L_{T_{24}}^{\times}$ (acting by left multiplication) and
$r\in T_{q}$, then the action $(\l;r)\,x\to\l x+r$ is biholomorphic
on $T_{q}$ so that we have the natural embedding $\L_{T_{24}}^{\times}\ltimes\mathbb{A}_{0}(T_{q})\subset\mathbb{A}(T_{q})$.
Let us define the sub-groups $\hat{Q}_{8}$ and $\hat{T}_{24}$ of
$\mathbb{A}(T_{q})$ as follows:
\[
\hat{Q}_{8}=\{1,\pm i,\pm j',\pm k'\}
\]
for $j'=(j;\a)$, $k'=(k;\a)$ where $\a=\frac{1}{2}(1+i)$ and 
\[
\hat{T}_{24}=\la\hat{Q}_{8},\,(t;\frac{1}{2}s)\ra,\,\,\,\text{for }s=\frac{1}{2}(1+i-j+k);
\]
(we recall that $t=\frac{1}{2}(1+i+j+k)$) thus by definition $\hat{Q}_{8}\subset\hat{T}_{24}$.
For $q\in\mathcal{X}$, the group $\hat{T}_{24}$ acts symplectically
on the torus $T_{q}=\HH/\L_{\hat{T}_{24}}$. That action is without
global fix points and so is the action of the sub-group $\hat{Q}_{8}\subset\hat{T}_{24}$.
One has $\hat{Q}_{8}\simeq Q_{8}$ and $\hat{T}_{24}\simeq T_{24}$
as abstract groups.
\begin{thm}
\label{thm:(Fujiki-)-Translated}(Fujiki \cite[Theorem 3.17]{Fujiki})
Let $G$ be a reduced finite group acting symplectically on a torus
$T$, such that there is no global fixed point.\\
The group $G$ is isomorphic to $Q_{8}$ or $T_{24}$. If $G\simeq Q_{8}$
(resp. $T_{24}$), then there exists $q\in\mathcal{X}$ such that
$(T,G)$ is isomorphic to $(T_{q},\hat{Q}_{8})$ (resp. $(T_{q},\hat{T}_{24})$),
where in both cases $T_{q}=\HH/\L_{24}$.
\end{thm}
\begin{rem}
A) Any action of $Q_{8}$ on the torus $\HH/\L_{Q_{8}}$ has a global
fix point. \\
B) By \cite[Proposition 5.7 p 62]{Fujiki}, the complex torus $T_{q}$
is algebraic if and only if $\exists\mu\in\RR,\,\mu q\in\text{\ensuremath{\L}}_{G}$.
There are an infinite number of such $q\in X$. Moreover if $T_{q}$
is algebraic, it has maximal Picard number. 
\end{rem}
On the following table are summarized the $10$ $ADE$ configurations
on generalized Kummer surfaces:\vspace{2mm}

\begin{tabular}{|c|c|c|c|}
\hline 
Configuration & Groups & References for $K_{G}$ & $\rho$\tabularnewline
\hline 
$16A_{1}$ & $\ZZ/2\ZZ$ & \cite{Nikulin, Morrison} & $16$\tabularnewline
\hline 
$9A_{2}$ & $\ZZ/3\ZZ$ & \cite{Bertin} & $18$\tabularnewline
\hline 
$6A_{1}+4A_{3}$ & $\ZZ/4\ZZ$ & \cite{Bertin} & $18$\tabularnewline
\hline 
$5A_{1}+4A_{2}+A_{5}$ & $\ZZ/6\ZZ$ & \cite{Bertin} & $18$\tabularnewline
\hline 
$2A_{1}+3A_{3}+2D_{4}$ & $Q_{8}$  & \cite[\textsection 4.2.2]{Garbagnati}, \cite[Prop. 2.1]{Wendland} & $19$\tabularnewline
\hline 
$3A_{1}+4D_{4}$ & $Q_{8}\subset T_{24}$  & \cite[\textsection 4.2.3]{Garbagnati} & $19$\tabularnewline
\hline 
$A_{1}+6A_{3}$ & $\hat{Q}_{8}$  &  & $19$\tabularnewline
\hline 
${\color{black}A_{1}+2A_{2}+3A_{3}+D_{5}}$ & $Q_{12}$ & \cite[\textsection 4.2.5]{Garbagnati} & $19$\tabularnewline
\hline 
${\color{black}A_{1}+4A_{2}+D_{4}+E_{6}}$ & $T_{24}$ & \cite[\textsection 4.2.4]{Garbagnati}, \cite[Prop. 2.1]{Wendland} & $19$\tabularnewline
\hline 
$4A_{2}+2A_{3}+A_{5}$ & $\hat{T}_{24}$  &  & $19$\tabularnewline
\hline 
\end{tabular}

The column $\rho$ gives the contribution of the given configuration
of $(-2)$-curves to the Picard number of the K3 surface. \\
About generalized Kummer surfaces, one must cite the work of Enriques
and Severi \cite{Enriques} more than one century ago, who were the
first to study generalized Kummer surfaces obtained as quotients of
Jacobians of curves. They saw the 10 cases of the above table. They
also described the resulting singularities (with errors for some non-cyclic
groups). \\
In \cite{Cinkir} \c Cinkir and \"Onsiper study generalized Kummer
surfaces and describe the quotient singularities (but some cases are
missing). In \cite{Onsiper} \"Onsiper and Sert\"oz give a generalization
of Shioda-Inose structures to these generalized Kummer surfaces. In
\cite{Bertin}, Bertin describes the primitive sub-lattices containing
the configurations for cyclic groups $\ZZ_{n}$, $n\in\{3,4,6\}$,
after the work of Nikulin \cite{Nikulin} (and Morrison \cite{Morrison})
for $n=2$. In \cite{Wendland}, Wendland studies that problem for
some non-cyclic groups preserving globally a point, a work which was
later corrected and completed by Garbagnati in \cite{Garbagnati}.\\
 To be more exhaustive, one must also mention that Fujiki studied
the possible $ADE$ singularities in \cite{Fujiki}, and Katsura \cite{Katsura}
worked out the possible symplectic groups in characteristic $>0$,
illustrating each cases by examples. 

\subsection{The configuration $\hat{Q}_{8}:\,\,A_{1}+6A_{3}$\label{subsec:Q8}}

Let $X$ be the K3 surface obtained as the desingularization of the
quotient $T_{q}/\hat{Q}_{8}$ of a torus $T_{q}=\HH/\L_{T_{24}}$
by the action of $\hat{Q}_{8}\subset\hat{T}_{24}$ described in Section
\ref{subsec:Fujiki's-constructions}.
\begin{lem}
The singularities of the quotient surface $T_{q}/\hat{Q}_{8}$ are
$A_{1}+6A_{3}.$ 
\end{lem}
\begin{proof}
Since the square of any order $4$ elements in $Q_{8}$ is the multiplication
by $-1$ map $[-1]_{T}$, the fixed point sets of these elements are
included into the fixed point set of $[-1]_{T}$ i.e. the set of $2$-torsion
points of $T_{q}$. For $a,b,c,d\in\{0,1\}$, let us denote by $abcd$
the $2$-torsion point $\frac{a}{2}+\frac{b}{2}i+\frac{c}{2}j+\frac{d}{2}t\in T_{q}$.
One has $i(0011)=0101$, $i(1001)=1111$, $i(0001)=1011$, $i(0111)=1101$,
$j'(0000)=1100$, $j'(1010)=0110$ etc, and we obtain that the fixed
point set of the order $4$ elements $i,\,j'=(j;\a),\,k'=(k;\a)$
(where $\a=\frac{1}{2}(1+i)$) of $\hat{Q}_{8}$ are 
\[
\begin{array}{c}
\fix i)=\{0000,\,1100,\,1010,\,0110\}\,\\
\fix j')=\{0011,\,0101,\,1001,\,1111\}\,\\
\fix k')=\{0001,\,1011,\,0111,\,1101\}.
\end{array}
\]
Using that $k'=ij',\,j'=-ik'$ etc, we compute that on the quotient
surface there are $2A_{3}$ which are the images of $\fix i)$, $2A_{3}$
images of $\fix j')$ and $2A_{3}$ images of $\fix k')$; the image
of the $4$ remaining $2$-torsion points in $T_{q}$ (the orbit of
$1000$) is a $A_{1}$. 
\end{proof}
Let now $X$ be any $K3$ surface containing a configuration $A_{1}+6A_{3}$.
For $1\leq r\leq6$, we denote by 
\[
C_{r}^{s},\,1\leq s\leq3
\]
the resolution of the $6A_{3}$, where $C_{r}^{1}C_{r}^{2}=C_{r}^{2}C_{r}^{3}=1$
and the other intersection numbers among the curves $C_{r}^{s}$ are
$0$ or $-2$. Let $C_{0}$ be the resolution of the $A_{1}$. The
discriminant group of the lattice $F_{\hat{Q}_{8}}$ generated by
the curves $C_{r}^{s}$, $1\leq r\leq6$, $s\in\{1,2,3\}$ and $C_{0}$
is $\ZZ_{2}\times(\ZZ_{4})^{6},$ it is generated by $t_{0}=\frac{1}{2}C_{0}$
and 
\[
t_{r}=\frac{1}{4}(C_{r}^{1}+2C_{r}^{2}+3C_{r}^{3}),\,r\in\{1,\dots,6\}.
\]
Let $K_{\hat{Q}_{8}}$ be the Kummer lattice of $\hat{Q}_{8}$: the
minimal primitive sub-lattice of $\NS(X)$ containing the lattice
$F_{\hat{Q}_{8}}$.
\begin{prop}
\label{prop:The-lattice-K8}The lattice $K_{\hat{Q}_{8}}$ is generated
by $F_{\hat{Q}_{8}}$ and by the divisors
\[
\d_{1}=(1,1,1,1,2,0),\,\d_{2}=(1,3,2,0,1,3)
\]
in the base $t_{1},\dots,t_{6}$ (up to reorder the $t_{r}$ and $C_{r}^{i}$).
\\
The lattice $K_{\hat{Q}_{8}}$ has discriminant group $\ZZ_{2}\times(\ZZ_{4})^{2}$;
the index of $F_{\hat{Q}_{8}}$ in $K_{\hat{Q}_{8}}$ equals $16$.
\end{prop}
\begin{proof}
The curves $C_{r}^{1},\,C_{r}^{3},\,r\in\{1,\dots,6\}$ and $C_{0}$
form a configuration of $13$ disjoint $A_{1}$. Therefore there exists
two linearly independent even sets supported on $12$ of these curves.
The curve $C_{0}$ cannot be part of such an even set (see Remark
\ref{rem:necessary cond even}). Therefore, up to permuting the indices,
the three even sets are 
\[
\begin{array}{c}
v_{1}=C_{1}^{1}+C_{1}^{3}+C_{2}^{1}+C_{2}^{3}+C_{3}^{1}+C_{3}^{3}+C_{4}^{1}+C_{4}^{3}\\
v_{2}=C_{3}^{1}+C_{3}^{3}+C_{4}^{1}+C_{4}^{3}+C_{5}^{1}+C_{5}^{3}+C_{6}^{1}+C_{6}^{3}\\
v_{3}=C_{1}^{1}+C_{1}^{3}+C_{2}^{1}+C_{2}^{3}+C_{5}^{1}+C_{5}^{3}+C_{6}^{1}+C_{6}^{3}
\end{array}
\]
and $\frac{1}{2}v_{1},\frac{1}{2}v_{2},\frac{1}{2}v_{3}$ are in fact
element of $\NS(X)$. In the discriminant group of $F_{\hat{Q}_{8}}$,
one has $\frac{1}{2}v_{1}+\frac{1}{2}v_{2}=\frac{1}{2}v_{3}$. Let
us denote by $L_{8}$ the lattice spanned by $F_{\hat{Q}_{8}}$ and
$\frac{1}{2}v_{1},\frac{1}{2}v_{2},\frac{1}{2}v_{3}$. The discriminant
group of $L_{8}$ is $(\ZZ/2\ZZ)^{5}\times(\ZZ/4\ZZ)^{2}$, of  length
$7>\rk(L_{8}^{\perp})=3$, thus there exists other divisibilities.
\\
Since a set of $12$ disjoint $A_{1}$ contains at most two linearly
independent even sets, these are divisibilities by $4$. Comparing
the length, one obtain that there are two linearly independent $4$-divisible
classes. In $A_{1}+6A_{3}$ there is no sub-configuration $4A_{3}+6A_{1}$
that could come from the quotient of a torus by an order $4$ automorphism.
A quotient of a K3 by an order $4$ automorphism has singularities
$4A_{3}+2A_{1}$ and the $4\times2$ disjoint configurations $A_{1}$
supported on the sub-configuration $4A_{3}$ of $4A_{3}+2A_{1}$ must
be divisible by $2$. Thus in our configuration $6A_{3}+A_{1}$, these
$4A_{3}$ are (supported on $4$ times the following elements): 
\[
t_{1},t_{2},t_{3},t_{4}\text{ or\,}\,t_{1},t_{2},t_{5},t_{6}\text{ or }t_{3},t_{4},t_{5},t_{6}.
\]
Once the $4A_{3}$ are chosen, there are two choices for the $2A_{1}$
such that $4A_{3}+2A_{1}$ becomes $4$-divisible: one can take two
disjoint curves in the resolution of the $5^{th}$ or of the $6^{th}$
$A_{3}$'s. Up to permuting the $t_{j}$, and also since one has some
freeness to permute $C_{r}^{1}$ with $C_{r}^{3}$, one can suppose
that 
\[
\d_{1}=(1,1,1,1,2,0)
\]
 (written in the canonical base of the subgroup $\ZZ^{6}\subset F_{\hat{Q}_{8}}^{\vee}$
generated by the $t_{i},\,i\in\{1,\dots,6\}$) is integral. The relations
$\d_{1}\d_{2}\in\ZZ$, $\d_{2}^{2}\in2\ZZ$ forces the other generator
$\d_{2}$ supported on $t_{1},t_{2},t_{5},t_{6}$ to be $\d_{2}=(1,3,2,0,1,3)$
(or $(3,1,2,0,3,1)$, but both generates the same group in the discriminant
group). Then $\d_{3}=(2,0,3,1,3,3)$ is supported on $t_{3},t_{4},t_{5},t_{6}$
and equals $\d_{1}+\d_{2}$ in the discriminant group.\\
  The lattice generated by $F_{\hat{Q}_{8}}$ and the $\d_{i},\,i\in\{1,2,3\}$
has discriminant group $\ZZ_{2}\times(\ZZ_{4})^{2}$ (of length $3$).
Another divisibility by $2$ is not possible because a set of $13$
disjoint $A_{1}$ support at most two linearly independent even sets.
If there was another independent $4$-divisible set, it would create
other even sets. Therefore that lattice is primitive and equals $K_{\hat{Q}_{8}}$.
\end{proof}
\begin{rem}
Let $Y$ be the K3 associated to the $\ZZ_{4}$-cover defined by $\d_{1}$.
There exists on $Y$ a configuration $4A_{3}+6A_{1}$, therefore $Y=\Km(T',\ZZ_{4})$
for some torus $T'$. The order $4$ automorphism $\tau$ such that
$Y/\tau$ is birational to $X$ lifts to an automorphism of $T'$.
The group generated by the lifts and the automorphism $\s\in\aut(T')$
such that $Y$ is birational to $T'/\s$ has order $16$. Thus by
the classification of Fujiki, it contains a translation.
\end{rem}
Let us now prove the following result:
\begin{prop}
\label{prop:GroupQ8}Let $X$ be a $K3$ surface containing a configuration
$A_{1}+6A_{3}$. Then there exists $q\in\mathcal{X}$ such that $X=\Km(T_{q},\hat{Q}_{8})$
where $T_{q}=\HH/\frak{a}$.
\end{prop}
\begin{proof}
By the proof of Proposition \ref{prop:The-lattice-K8}, there exist
two linearly independent even sets which are supported on the $12$
disjoint rational curves of the sub-configuration $6A_{3}$ in $A_{1}+6A_{3}$.
\\
Taking the associated bidouble cover and its minimal model, the pull-back
of the $A_{1}$ and central curves in the six $A_{3}$ is a set of
$16$ disjoint $A_{1}$ curves $C_{i}$, this is therefore a Kummer
surface $\Km(T)$. Since the automorphisms in the group $(\ZZ/2\ZZ)^{2}$
preserves the branch locus $\sum_{1}^{16}C_{i}$, these automorphisms
lift to automorphisms of $T$. By the classification of Fujiki that
group must be isomorphic to $Q_{8}$ and the result follows from Fujiki's
classification Theorem \ref{thm:(Fujiki-)-Translated}.
\end{proof}
Let now $X$ be a K3 surface such that there exists a primitive embedding
of $K_{\hat{Q}_{8}}$ into $\NS(X)$. 
\begin{thm}
There exists a complex torus $T$ and a group of automorphism $G\simeq Q_{8}$
such that $X=\Km(T,G)$.
\end{thm}
\begin{proof}
Using Magma, one computes that the number of roots in $K_{\hat{Q}_{8}}$
($37$ of such) equals the number of roots of $F_{\hat{Q}_{8}}$.
Thus by Lemma \ref{prop: Bertin-Garbagati}, there exists a configuration
$A_{1}+6A_{3}$ of smooth irreducible rational curves. We then apply
Proposition \ref{prop:GroupQ8} and Remark \ref{rem:SimpleBase}.
\end{proof}

\subsection{The configuration $\hat{T}_{24}:\,\,4A_{2}+2A_{3}+A_{5}$\label{subsec:The-configuration Q24}}

Let $X=\Km(T,\hat{T}_{24})$ be a K3 surface obtained as the desingularization
of the quotient of a complex torus $T=T_{q}$ by the action of $\hat{T}_{24}$.
\\
One computes that the order $3$ automorphism $w=(t;\frac{1}{2}s)^{2}$
fixes a unique $2$-torsion point on the torus $T_{q}=\HH/\L_{\hat{T}_{24}}$;
that point is not in the fixed point sets of the automorphisms $i,\,j'=(j,\a),\,k'=(k,\a)$.
The K3 surface $T_{q}/\hat{T}_{24}$ is a quotient of $T_{q}/\hat{Q}_{8}$
(where $\hat{Q}_{8}\subset\hat{T}_{24}$ is the unique normal subgroup
of order $8$) by the order $3$ automorphism $w'$ induced by $w$.
\\
An order $3$ automorphism on a smooth K3 has $6$ isolated fixed
points. In our situation, two of these fixed points are on the isolated
$A_{1}$ in $6A_{3}+A_{1}$, thus taking the resolution one gets an
$A_{5}$. The configurations $6A_{3}$ on $T_{q}/\hat{Q}_{8}$ are
permuted by $3$, creating $2A_{3}$ on the quotient surface, there
are moreover $4A_{2}$ coming from the $4$ other fixed points of
$w'$. We thus obtain:
\begin{lem}
\textup{The K3 surface $X=\Km(T,\hat{T}_{24})$ contains a configuration
$4A_{2}+2A_{3}+A_{5}.$} 
\end{lem}
Let now $X$ be any K3 surface containing a configuration $4A_{2}+2A_{3}+A_{5}$. 
\begin{prop}
\label{prop:T24prop19}There exists a torus $T$ with an action of
the group $\hat{T}_{24}$ such that $X=\Km(T,\hat{T}_{24})$.
\end{prop}
\begin{proof}
The configuration $4A_{2}+2A_{3}+A_{5}$ contains $8$ disjoint $A_{2}$
sub-configurations. The discriminant group of $8A_{2}$ is $(\ZZ_{3})^{8}$.
It has length $8>\min(16,22-16)=6$, therefore there exists a non-trivial
$3$-divisible class $D$ with support on $6$ of the $8A_{2}$. By
Remark \ref{rem:necessary cond even}, the support of $D$ is the
sub-configuration $6A_{2}$ contained in $4A_{2}+A_{5}$. \\
The surface associated to the triple cover branched on the support
of $D$ is a K3 surface $Y$ with a configuration $6A_{3}+A_{1}$
and having an order $3$ automorphism $\s$. We proved in Proposition
\ref{prop:GroupQ8} that the surface $Y$ is of type $Y=\Km(T,\hat{Q}_{8})$.
The automorphism $\s$ must preserves the $2$ linearly independent
even sets on $Y$ supported on the $6A_{3}$, otherwise there would
be other divisibilities relations. Therefore by Lemma \ref{lem:lift order 3 to bidoubles},
the automorphism $\s$ lifts to the $(\ZZ_{2})^{2}$-cover of $Y$,
which contains a $16A_{1}$ configuration. These $16A_{1}$ are pull-back
of curves in $X$, thus $\s$ lifts to an automorphism $\tilde{\s}$
of $T$, and $X$ is the Kummer surface associated to the group generated
by $\hat{Q}_{8}$ and $\tilde{\s}$, which has order divisible by
$3$. By Theorem \ref{thm:(Fujiki-)-Translated}, that group is $\hat{T}_{24}$.
\end{proof}
Let again $X$ be a K3 surface containing a configuration $4A_{2}+2A_{3}+A_{5}$.
The discriminant group of the lattice $F_{\hat{T}_{24}}$ generated
by the curves in $4A_{2}+2A_{3}+A_{5}$ is 
\[
(\ZZ_{3})^{4}\times(\ZZ_{4})^{2}\times\ZZ_{6}.
\]
It has length $5$. There exists an integral class $\g=\frac{1}{3}D$,
where $D$ is supported on the $6$ disjoint $A_{2}$ on the sub-configuration
$4A_{2}+A_{5}$ (see proof of Proposition \ref{prop:T24prop19}).
The discriminant group of the lattice generated by $\g$ and $F_{\hat{T}_{24}}$
is 
\[
(\ZZ_{3})^{2}\times(\ZZ_{4})^{2}\times\ZZ_{6}\simeq(\ZZ_{12})^{2}\times\ZZ_{6},
\]
which has length $3=22-19$. By Remark \ref{rem:necessary cond even},
there are no other set of $6$ disjoint $A_{2}$ which is $3$-divisible,
nor there are even sets, therefore we get the following result:
\begin{prop}
\label{prop:The-lattice-K24}The lattice generated by $F_{\hat{T}_{24}}$
and $\d$ is the minimal primitive sub-lattice $K_{\hat{T}_{24}}\subset\NS(X)$
containing $F_{\hat{T}_{24}}$. The discriminant group of $K_{\hat{T}_{24}}$
is $(\ZZ_{12})^{2}\times\ZZ_{6}$.
\end{prop}
Thus if $X=\Km(T,\hat{T}_{24})$, then there is a primitive embedding
of $K_{\hat{T}_{24}}$ into $\NS(X)$. Conversely, let $X$ be any
K3 surface:
\begin{thm}
Suppose that there is a primitive embedding of $K_{\hat{T}_{24}}$
into $\NS(X)$, then $X=\Km(T,\hat{T}_{24})$.
\end{thm}
\begin{proof}
Using MAGMA, it turns out that $K_{\hat{T}_{24}}$ has the same roots
as $F_{\hat{T}_{24}}$. We then apply Lemma \ref{prop: Bertin-Garbagati}
and Remark \ref{rem:SimpleBase}.
\end{proof}

\section{The case of equality in the orbifold Bogomolov Miyaoka Yau inequality
\label{sec:Kobayashi's-problem-for}}

\subsection{K3 surfaces}

For an orbifold $\mathcal{X}$ with only $ADE$ singularities, such
that $X$ has Kodaira dimension $0$ or $1$, one has $c_{1}^{2}(\mathcal{X})=0$
and the second orbifold Chern number is defined by $c_{2}(\mathcal{X})=c_{2}(X)-m(\mathcal{C})$,
where the rational number $m(\mathcal{C})\geq0$ depends only on the
type and number of the singularities of $\mathcal{X}$ (see Section
\ref{sec:Preliminaries}). For a K3 surface, part A) of Theorem \ref{thm:(Kobayashi)-1}
is thus equivalent to $m(\mathcal{C})\leq c_{2}(X)=24$. Our aim is
to characterize configurations $\mathcal{C}$ for which equality 
\[
c_{2}(\mathcal{X})=0
\]
holds, i.e.when $m(\mathcal{C})=c_{2}(X)$. For any configuration
$\mathcal{C}$ among the following $10$ configurations 
\[
\begin{array}{c}
16A_{1},\,\,\,9A_{2},\,\,\,6A_{1}+4A_{3},\,\,\,5A_{1}+4A_{2}+A_{5},\hfill\\
3A_{1}+4D_{4},\,\,2A_{1}+3A_{3}+2D_{4},\,\,A_{1}+2A_{2}+3A_{3}+D_{5},\\
A_{1}+4A_{2}+D_{4}+E_{6},\,\,\,4A_{2}+2A_{3}+A_{5},\,\,\,A_{1}+6A_{3},\qquad
\end{array}
\]
one has $m(\mathcal{C})=24$, moreover:
\begin{thm}
Suppose that a K3 $X$ contains the configuration $\mathcal{C}$ and
let $X\to\mathcal{X}$ be the contraction of the curves in $\mathcal{C}$.
There exists a finite group of automorphisms $G$ acting on a torus
$T$ such that $X=\Km(T,G)$ and $\mathcal{X}=T/G$.
\end{thm}
This is a result of Nikulin \cite{Nikulin} for $16A_{1}$, Bertin
\cite{Bertin} for the cases $9A_{2},\,6A_{1}+4A_{3},\,5A_{1}+4A_{2}+A_{5}$,
of Propositions \ref{prop:GroupQ8} and \ref{prop:T24prop19} for
the two last cases and Garbagnati \cite{Garbagnati} for the remaining
cases. A direct consequence is:
\begin{thm}
\label{thm:For-each-of}For each of the $10$ above cases, there exists
a lattice $\Gamma$ in the affine automorphism group of $\CC^{2}$
such that $\mathcal{X}=\CC^{2}/\Gamma$.
\end{thm}
In other words, each of the orbifold surfaces $\mathcal{X}$ is uniformisable
by $\CC^{2}$. 

It is easy to compute that there are $8$ other possible configurations
$\mathcal{C}$ with Milnor number $\rho\leq19$ (since $H^{2}(X,\ZZ)$
has signature $(3,19)$) and $m(\mathcal{C})=24$. These configurations
are: 
\[
\begin{array}{ccccc}
\mathcal{C}_{1}= & 11A_{1}+2A_{3}\hfill &  & \mathcal{C}_{5}= & 5A_{1}+A_{2}+D_{4}+D_{8}\quad\\
\mathcal{C}_{2}= & 7A_{1}+A_{3}+2D_{4}\hfill &  & \mathcal{C}_{6}= & 5A_{1}+A_{3}+A_{4}+D_{7}\quad\\
\mathcal{C}_{3}= & 5A_{1}+A_{3}+A_{7}+D_{4}\quad &  & \mathcal{C}_{7}= & 2A_{1}+2A_{2}+2D_{4}+D_{5}\\
\mathcal{C}_{4}= & 6A_{1}+2A_{2}+A_{3}+D_{5} &  & \mathcal{C}_{8}= & A_{1}+4A_{2}+2D_{5}.\hfill
\end{array}
\]
The aim of this section is to prove the following result, which with
Theorem \ref{thm:For-each-of} implies Theorem \ref{thm:MAIN-1}:
\begin{prop}
For any $i\in\{1,\dots,8\}$ there is no complex K3 surface containing
a configuration $\mathcal{C}_{i}$ .
\end{prop}
\begin{rem}
Some of these configurations $\mathcal{C}_{i}$ may exist in characteristic
$p>0$. Indeed by \cite[Corollary 3.17 and Remark 7.3]{Katsura},
the cyclic groups $\ZZ_{5},\ZZ_{8},\ZZ_{10},$ $\ZZ_{12}$, the binary
dihedral groups $\DD_{n-2}$ (of order $4n-8$, creating singularity
$D_{n}$) with $n\in\{4,\dots,8\}$ and the binary octahedral and
icosahedral groups act symplectically on some Abelian surfaces in
characteristic $p>0$. 
\end{rem}

\subsubsection{Configuration $\mathcal{C}_{1}=11A_{1}+2A_{3}$}

The $11A_{1}$ plus one curves from each $A_{3}$ form a set of
$13$ disjoint curves. By Remark \ref{rem:necessary cond even} there
are two linearly independent even sets supported on $12$ curves.
This is impossible since a unique curve on a $A_{3}$ cannot be part
of an even set. Such a configuration $\mathcal{C}_{1}=11A_{1}+2A_{3}$
does not exist on a complex K3.

\subsubsection{Configuration \textup{$\mathcal{C}_{2}=7A_{1}+A_{3}+2D_{4}$}. }

There are $14$ disjoint rational curves: $7A_{1}$, one curve in
$A_{3}$ plus $3$ curves for each $D_{4}$. If there are $14$ disjoint
rational curves on a K3 surface, there are three independent even
sets, supported on all the curves. But one curve in $A_{3}$ can not
be in the support of an even set. Therefore that configuration does
not exist on a complex K3 surface.

\subsubsection{Configuration $\mathcal{C}_{3}=5A_{1}+A_{3}+A_{7}+D_{4}$. }

Let us consider the following set of $12$ disjoint rational curves
supported on $\mathcal{C}_{3}$ : $5A_{1}$ plus the two disjoint
curves in $A_{3}$, plus two disjoint curves in $A_{7}$ (at the extrema)
and three disjoint curves in $D_{4}$. It contains an even set of
curves $E$. The two curves on the $A_{7}$ cannot be on the support
of $E$. One must take $0$ or $2$ curves in the $D_{4}$, thus the
even set is made of $4A_{1}$ plus the two disjoint curves on the
$A_{3}$ and two disjoint curves on the $D_{4}$. The K3 double cover
will have a configuration 
\[
2A_{1}+A_{1}+2A_{7}+A_{3},
\]
but it would have Picard number $>20$, contradiction.

\subsubsection{Configuration $\mathcal{C}_{4}=6A_{1}+2A_{2}+A_{3}+D_{5}$. }

There is a set of $13$ disjoint rational curves on $\mathcal{C}$.
There must be two linearly independent even set supported on $12$
of these curves. But an even set cannot contain the curves in a $A_{2}$. 

\subsubsection{Configuration\textup{ $\mathcal{C}_{5}=5A_{1}+A_{2}+D_{4}+D_{8}$}}

Let us consider the following $14$ curves: $5A_{1}$ plus one curve
in $A_{2}$, plus the $3$ disjoint curves in $D_{4}$ and the $5$
disjoint curves in $D_{8}$. As for configuration $\mathcal{C}_{2}$,
there are three independent even sets, supported on all the curves.
But the curve in $A_{2}$ cannot be in the support of a 2-divisible
even set.

\subsubsection{Configuration\textup{ $\mathcal{C}_{6}=5A_{1}+A_{3}+A_{4}+D_{7}$}}

The sub-configuration $4A_{1}+A_{3}+A_{4}+D_{7}$ contains $13$ disjoint
rational curves (the $5A_{1}$ plus $2$ disjoint curves in $A_{3}$,
$2$ in $A_{4}$ and $4$ curves in $D_{7}$), thus there exist two
independents even sets supported on $12$ curves. However the $2$
disjoint curves in $A_{4}$ cannot be part of an even set.

\subsubsection{Configuration $\mathcal{C}_{7}=2A_{1}+2A_{2}+2D_{4}+D_{5}$}

There are $13$ disjoint rational curves on $\mathcal{C}_{7}$, thus
there exists two linearly independent even sets and we obtain a contradiction
as before by looking at the possible supports for these two even sets.

\subsubsection{Configuration $\mathcal{C}_{8}=A_{1}+4A_{2}+2D_{5}$}

The discriminant group of $\mathcal{C}_{8}$ is 
\[
\ZZ_{2}\times(\ZZ_{3})^{2}\times(\ZZ_{12})^{2}
\]
it has length $5$, but the lattice has rank $19$ and a minimal primitive
sub-lattice of rank $19$ has a discriminant with length at most $3$.
Thus there exists some divisibilities by $2$ or $3$. But it is easy
to check using Remark \ref{rem:necessary cond even} that no such
an even set can exist, nor there exists a $3$-divisible set of $6A_{2}$. 

\subsection{Enriques surfaces}

An Enriques surface $Z$ has invariants $K_{Z}^{2}=0,\,c_{2}=12$
with $2K_{Z}=0$. It is the quotient of a K3 by a fix-point free involution.
Let $\mathcal{C}$ be a configuration of $ADE$ curves on an Enriques
surface $Z$ such that the associated orbifold $\mathcal{Z}$ has
Chern numbers $c_{1}^{2}(\mathcal{Z})=3c_{2}(\mathcal{Z})$.
\begin{prop}
\label{prop:8A1 enrik}The configuration $\mathcal{C}$ is $\mathcal{C}=8A_{1}$.
There exists an Abelian surface $A$ isogeneous to the product of
two elliptic curves, a group of automorphisms $G\simeq(\ZZ_{2})^{2}$
of the surface $A$ generated by the involution $[-1]$ and a fix-point
free involution such that $Z$ is the minimal resolution of $A/G$.
\end{prop}
\begin{proof}
For an Enriques surface the condition $c_{1}^{2}(\mathcal{Z})=3c_{2}(\mathcal{Z})$
is equivalent to $c_{2}(\mathcal{Z})=0$ i.e. $m(\mathcal{C})=12$.
\\
Let $X\to Z$ be the étale double cover of $Z$. The K3 surface $X$
contains the configuration $2\mathcal{C}$, which verifies $m(2\mathcal{C})=24$,
thus the only possibilities are $\mathcal{C}=8A_{1}$ and $\mathcal{C}=3A_{1}+2A_{2}$.
\\
Let $\s$ be the Enriques involution on $X$ so that $Z=X/\s$. The
involution $\s$ preserves the $16A_{1}$ (resp $6A_{1}+4A_{2}$)
on $X$, thus it lifts to an automorphism $\s'$ on the Abelian surface
$A$ such that $X=\Km(A)$ (resp $X=\Km(A,\ZZ_{4})$). Since $\s$
has no fix points on $X$, $\s'$ has no fix points either on $A$.

Let us study the case $\mathcal{C}=8A_{1}$. Suppose that a lift $\s'$
of $\s$ has order $4$; then $\s'^{2}$ is the transformation of
the double cover $A\to X$, i.e. $\s'^{2}=[-1]$. Since $H^{0}(Z,K_{Z})=0$,
$\s'$ must not preserve the space $H^{0}(A,K_{A})$, thus (up to
replacing $\s'$ by $\s'^{3}$) the eigenvalues of the analytic representation
of $\s'$ are $i,i$ and $A$ is the surface $(\CC/\ZZ[i])^{2}$,
$\s'$ is the multiplication by $i$ map composed by some translation.
But such morphism has always fixed points. \\
Therefore $\s'$ has order $2$, commutes with $[-1]$ and the eigenvalues
of its analytic representation are $(1,-1)$. Then there exists coordinates
of $T_{A}\simeq\CC^{2}$ such that $\s':A\to A$ is given by 
\[
\s'(z_{1},z_{2})=(-z_{1},z_{2})+v
\]
 where $v\in A$. Thus there exist a product $E_{1}\times E_{2}$
of elliptic curves and an isogeny $E_{1}\times E_{2}\to A$; moreover
since $\s'$ commutes with $[-1]$, one must have $v=-v$ i.e. $v$
is a $2$ torsion point. Since $\s'$ has no fix points $v$ is non-trivial.

Let us study the case $3A_{1}+2A_{3}$.  Suppose there exists an
Abelian surface $A$ and a group $G$ of order $8$, such that $A/G$
is an Enriques surface with a configuration $2A_{3}+3A_{1}.$ For
an automorphism $\tau$ let $\tau_{0}$ be the linear part of $\tau$,
and $G_{0}$ the group $\{\tau_{0}\,|\,\tau\in G\}$. An element $\tau$
in the kernel $K$ of $G\to G_{0}$ is a translation, but then $A/G$
is the surface $A'/G_{0}$ where $A'=A/K$ and $G_{0}$ has order
$4$, which lead to a contradiction. The group $G$ is therefore isomorphic
to $G_{0}$ and since it contains an order $4$ element, it is among
the following groups 
\[
\ZZ_{8},\,\ZZ_{4}\times\ZZ_{2},\,\mathbb{D}_{4},\text{ or }Q_{8}.
\]
There are no order $8$ automorphisms acting on an Abelian surface
\cite{Fujiki} thus $G\neq\ZZ_{8}$. The group $G$ is generated by
a fix-point free involution $\s$ and an automorphism $\mu$ of order
$4$ such that $A/\mu$ is a K3 with $4A_{3}+6A_{1}$ (in particular
$\mu^{2}=[-1]$), moreover the involution $\s$ induces a fix-point
free non-symplectic involution on $A/\mu.$ The group $G$ is not
$Q_{8}$ since that groups has a unique involution. \\
Suppose that this is $\ZZ_{4}\times\ZZ_{2}=\la\mu\ra\times\la\s\ra$,
then $\s_{0}$ is $\s_{0}(z_{1},z_{2})=(-z_{1},z_{2})$ and 
\[
\s(z_{1},z_{2})=(-z_{1},z_{2})+v
\]
where $v$ is a non-trivial $2$-torsion point. Moreover since $\mu_{0}\s_{0}=\s_{0}\mu_{0}$,
the element $\mu$ must act diagonally, thus 
\[
\mu(z_{1},z_{2})=(iz_{1},-iz_{2}).
\]
Therefore $A=C\times C$, where $C$ is the elliptic curve $\CC/\ZZ[i]$.
One has 
\[
\s\mu(z_{1},z_{2})=(-iz_{1},-iz_{2})+v,
\]
which has always some fix points, creating $\frac{1}{4}(1,1)$ singularities,
but there is no such singularities on Enriques surfaces.\\
It remains the dihedral group $\mathbb{D}_{4}$ of order $8$. There
is only one faithful $2$-dimensional representation of $\mathbb{D}_{4}$,
which is generated by 
\[
\s_{0}(z_{1},z_{2})=(-z_{1},z_{2}),\,\mu(z_{1},z_{2})=(-z_{2},z_{1}).
\]
Taking $\s(z_{1},z_{2})=(-z_{1},z_{2})+v$ where $v$ is a $2$-torsion
point, the involution $\s\mu$ has a one dimensional fixed point set,
thus the quotient of $A$ by $\DD_{4}$ is a rational surface (see
e.g. \cite{Katsura}). We thus proved that there is no Enriques surface
containing a configuration $3A_{1}+2A_{3}$.
\end{proof}
\begin{example}
(Lieberman, see \cite{Keum}). Let $A$ be the product of two elliptic
curves $A=E_{1}\times E_{2}$ and let  $(e_{1},e_{2})$ be a $2$-torsion
point on $A$, with $e_{1}\neq0,\,e_{2}\neq0$. Then the endomorphism
$\tau:A\to A$ given by 
\[
\tau(z_{1},z_{2})=(-z_{1}+e_{1},z_{2}+e_{2})
\]
induces a fix-point free involution on the Kummer surface $\Km(A)$.
The associated Enriques surface contains a $8A_{1}$ configuration.
\end{example}
Let $T_{0}$ be a sub-group of torsion points on $A=E_{1}\times E_{2}$
as above, such that $\tau(T_{0})=T_{0}$ and $(e_{1},0),\,(0,e_{2})$
are not element of $T_{0}$. Then $\tau$ induces a fix-point free
involution $\tau'$ on the quotient $A/T_{0}$ and $A/\la T_{0},[-1],\tau\ra$
is an Enriques surface containing $8A_{1}$. Reciprocally, from the
proof of Proposition \ref{prop:8A1 enrik}, every Enriques surface
containing $8A_{1}$ is obtained by that construction.

\noindent Xavier Roulleau,
\\Aix-Marseille Universit\'e, CNRS, Centrale Marseille,
\\I2M UMR 7373,  
\\13453 Marseille, France
\\ {\tt Xavier.Roulleau@univ-amu.fr}

\vspace{0.3cm}

\begin{thebibliography}{99} 


\bibitem{Barth} Barth W., K3 surfaces with nine cusps, Geom. Dedicata 72 (1998) 171--178.

\bibitem{Bertin} Bertin J., Réseaux de Kummer et surfaces K3, Invent. Math. 93, No.2, 267--284 (1988).


\bibitem{Catanese} Catanese F., Singular bidouble covers and the construction of interesting algebraic surfaces, Proc. of the alg. geom. conf. in honor of F. Hirzebruch's 70th birthday, AMS. Contemp. Math. 241, 97--120 (1999).

\bibitem{Cinkir} \c Cinkir Z., \"Onsiper H., On symplectic quotients of K3 surfaces, Indag. Math., New Ser. 11, No.4, 533--538 (2000).

\bibitem{Enriques} Enriques F., Severi Fr., Mémoire sur les surfaces hyperelliptiques, Parties I et II, Acta Math. 32, 283--392 (1909) and Acta Math. 33, 321--403 (1910).

\bibitem{Fujiki} Fujiki A., Finite automorphism group of complex Tori of dimension two, Publ. Res. Inst. Math. Sci. 24, No.1, 1--97 (1988).

\bibitem{Garbagnati} Garbagnati A., On K3 surface quotients of K3 or Abelian surfaces, Can. J. of Math.  69 (2017), 338--372.

\bibitem{GS} Garbagnati A., Sarti A., Kummer surfaces and K3 surfaces with $(\ZZ/2\ZZ)^4$ symplectic action, Rocky Mount. J. of Math., 46 (2016), vol. 4, 1141--1206

\bibitem{Griess} Griess R., An introduction to groups and lattices: finite groups and positive definite rational lattices,  Advanced Lectures in Mathematics 15, International Press, Somerville, MA; Higher Education Press, Beijing, 2011. iv+251 pp.

\bibitem{Hirzebruch} Hirzebruch F., Singularities of Algebraic surfaces and characteristic numbers, Algebraic geometry, Proc. Lefschetz Centen. Conf., Mexico City/Mex. 1984, Part I, Contemp. Math. 58, 141--155 (1986).

\bibitem{Katsura} Katsura T., Generalized Kummer surfaces and their unirationality in characteristic $p$, J. Fac. Sci., Univ. Tokyo, Sect. I A 34, 1--41 (1987).

\bibitem{Keum} Keum J.H., Every algebraic Kummer surface is the K3-cover of an Enriques surface, Nagoya Math. J. 118, 99--100 (1990).

\bibitem{Kobayashi} Kobayashi R., Uniformization of complex surfaces. Kähler metric and moduli spaces, 313--394, Adv. Stud. Pure Math., 18--II, Academic Press, Boston, MA, 1990

\bibitem{Kobayashi2} Kobayashi R., Nakamura S., Sakai F., A Numerical Characterization of Ball Quotients for Normal Surfaces with Branch Loci, Proc. Japan Acad., 65, Ser. A (1989)

\bibitem{Morrison} Morrison D., On K3 surfaces with large Picard number, Invent. Math. 75, 105--121 (1984)

\bibitem{Megyesi} Megyesi G., Generalisation of the Bogomolov-Miyaoka-Yau inequality to singular surfaces. Proc. London Math. Soc. (3) 78 (1999), no. 2, 241--282. 


\bibitem{Nikulin} Nikulin V., Finite automorphism groups of K\"ahler K3 surfaces, Trans. Moscow Math. Soc. 38 (1980) 71--135.

\bibitem{Nikulin2} Nikulin V., Integral symmetric bilinear forms and some applications, Math. USSR, Izv. 14, 103--167 (1980).

\bibitem{Onsiper} \"Onsiper H., Sert\"oz S.,  Generalized Shioda-Inose structures on K3 surfaces, Turk. J. Math. 23, No.4, 575--578 (1999).

\bibitem{PRR} Polizzi F., Rito C., Roulleau X.,  A pair of rigid surfaces with $p_g =q=2$and $K^2 =8$ whose universal cover is not the bidisk, arXiv 1703.10646

\bibitem{RR} Roulleau X., Rousseau E., On the hyperbolicity of surfaces of general type with small $c_1^2$, J. Lond. Math. Soc., II. Ser. 87, No. 2, 453--477 (2013).

\bibitem{Schutt} Sch\"utt M., Divisibilities among nodal curves, preprint arXiv 1706.00570 



\bibitem{Wendland} Wendland K., Consistency of orbifold conformal field theories on K3, Adv. Theor. Math. Phys. 5, No. 3, 429--456 (2001). 


\end{thebibliography}
\end{document}